\documentclass[12pt]{amsart}
\usepackage[utf8]{inputenc}
\usepackage{amsmath,amsthm,amssymb}
\usepackage{xcolor}

\usepackage[british]{babel}

\usepackage[shortlabels]{enumitem}
\usepackage{hyperref}
\usepackage{mathtools}
\usepackage{graphicx}

\newtheorem{theorem}{Theorem}[section]

\newtheorem{remark}[theorem]{Remark}
\newtheorem{lemma}[theorem]{Lemma}

\newtheorem{prop}[theorem]{Proposition}

\newtheorem{corollary}[theorem]{Corollary}
\mathtoolsset{showonlyrefs}

\newcounter{stepctr}
\newif\ifinsteppage

\newenvironment{steps}{%
	\setcounter{stepctr}{0}%
	\insteppagefalse
}{%
	\par\bigskip
}

\newcommand{\step}[1]{%
	\ifinsteppage
	\par\bigskip
	\fi
	\insteppagetrue
	\refstepcounter{stepctr}%
	\noindent\textbf{Step \thestepctr:} #1.\par
}

\newcommand{\les}{\lesssim}

\newcommand{\alp}{\alpha}

\newcommand{\De}{\Delta}

\newcommand{\dem}{\de^{-O(\e)}}
\newcommand{\dep}{\de^{O(\e)}}
\newcommand{\demm}{\de^{-\e}}
\newcommand{\depp}{\de^\e}

\newcommand{\R}{\mathbb{R}}\newcommand{\C}{\mathbb{C}}

\newcommand{\Z}{\mathbb{Z}}

\newcommand{\vi}{\mathbf{i}}

\newcommand{\cA}{\mathcal{A}}
\newcommand{\cD}{\mathcal{D}}

\newcommand{\cU}{\mathcal{U}}

\newcommand{\dy}[2]{\cD_{#1}(#2)}

\newcommand{\cV}{\mathcal{V}}

\newcommand{\de}{\delta}
\newcommand{\e}{\epsilon}

\newcommand{\Q}{\mathbb{Q}}

\newcommand{\ups}{\upsilon}

\newcommand{\cns}[2]{N_{{#1}}({#2})}

\newcommand{\cn}[1]{N_\de({#1})}
\newtheorem*{ack*}{Acknowledgment}

\author{William O'Regan}
\begin{document}
	\title{Bourgain's projection theorem over normed division algebras}

	\thanks{WOR is supported in part by an NSERC Alliance grant administered by Pablo Shmerkin and Joshua Zahl}
	\begin{abstract}
      We give a simple and self-contained proof of an extension of a projection theorem of Bourgain over the reals to division algebras over local fields of zero characteristic. 
	\end{abstract}
	\maketitle

    \section{Introduction and statement of results}
    \subsection{Introduction}
    Bourgain's discretised projection theorem in $\R$ \cite{bou} is one of the most influential `expansion' results in analysis. In recent years, there have been a large number of applications in dynamical systems,  geometric measure theory and harmonic analysis. This includes: the radial projection problem \cite{OSW24, Ren23}, the Furstenberg set problem in the plane \cite{OrponenShmerkin23b, RenWang23}, and the Kakeya problem in $\mathbb{R}^3$ \cite{WangZahl25a, WangZahl25b, WangZahl25c}. (There is also now a simple and essentially self-contained proof of his result \cite{o2025simple}.)

    To be more precise, none of these works invoke Bourgain's theorem directly, but rather build on the improved Furstenberg estimate in \cite{OrponenShmerkin23}, which in turn uses Bourgain's theorem as a key ingredient. Bourgain's projection theorem is also a key (indirect) component in recent progress on the Falconer distance set problem and related problems \cite{RazZahl24, ShmerkinWang25}.

Bourgain's projection theorem has also been a key tool in obtaining quantitative equidistribution estimates in ergodic theory. An early application was to the equidistribution of orbits of non-abelian semigroups on the torus \cite{BFLM11}. Bourgain's theorem is also a crucial component in the recent breakthrough on exponential equidistribution of random walks on quotients of $\textrm{SO}(2,1)$ and $\textrm{SO}(3,1)$ \cite{BenardHe24}, and the related work establishing Khintchine's theorem for self-similar measures \cite{BHZ24}. These recent works do not invoke Bourgain's theorem directly, but rather build on higher rank \cite{He20} and non-linear \cite{Shmerkin23} versions of it, which in turn apply Bourgain's projection theorem as a black box. 

Outside of the setting over $\R,$ a version of Bourgain's discretised projection theorem over $\C$ was obtained in \cite{bourgaingamspectral} with a direct application to the spectral gap problem for $SU(d).$ (The problem for $SU(2)$ was considered in \cite{bougam} and utilised a similar discretised sum-product estimate.) For other fields, or division algebras, it was obtained in \cite{nic} that if $E$ is a normed division algebra over $\R$ or $\Q_p$, and $A$ is a discretised set that behaves like a set of `dimension' $s,$ as well as is not trapped inside a sub-algebra, then $A+AA$ will grow in exponent. There, it is also shown that the growth of dimension may be made to be independent of the dimension of $E$ over its base field, provided that $E$ is not a non-commutative extension of $\Q_p.$ There, the strategy was inspired by the one of \cite{gut} which obtains a short proof of the discretised ring theorem. This, in turn, is inspired by a clever proof of the sum-product problem in $\mathbb{F}_p$ of Garaev \cite{garaev2007explicit}. 

The aim of this note is to combine the strategies of \cite{nic}, \cite{gut}, \cite{garaev2007explicit}, and \cite{o2025simple} to obtain a generalisation of Bourgain's projection theorem in normed division algebras over $\R$ or $\Q_p.$ Our proof is short, simple and uses few auxiliary results. While versions over $\R$ already see applications, we expect that the version over $\Q_p$ will see similar applications in the future.

    \subsection{Statement of results}
    Below is a generalisation of \cite[Theorem 2]{bou}. One should also view \cite[Proposition 2]{bourgaingamspectral} which gives a variant of this result over $\C.$ It is readily checked that Theorem \ref{thm.babyproj} implies that result. 
\begin{theorem}\label{thm.babyproj}
     Let $d$ be a positive integer. Let $E$ be a normed division algebra of dimension $d$ over $\Q_p$ or $\R.$ Let $0 < s \leq \sigma  < d, t > 0.$ There exists $c = c(s,\sigma,d,t)> 0$ so that the following holds for all $\de,\e >0$ small enough.
    
    Let $A \subset B(0,1)$ be a $(\de,s,\demm)$-set with $|A| = \de^{-\sigma}.$ Let $X \subset B(0,1)$ be a $(\de,t,\demm)$-set which avoids sub-algebras. Then there exists an $x \in X$ so that
    $$\cn{A + xA} > \de^{-c}|A|.$$
    
        Further, provided that $E$ is not a non-commutative division algebra over $\Q_p,$ then $c = c(s,\sigma,t,t/d).$ 
\end{theorem}
Notation and basic definitions may be found in Section \ref{subsect.defs}. Let $\pi_x$ be the map $(a,b)\mapsto a+xb.$
    Below is a generalisation of \cite[Theorem 3]{bou}, with a conclusion similar to that of the closely related \cite[Theorem 5]{bou}. Similarly to \cite{bou}, the proof involves reducing the statement of Theorem \ref{thm.adultproj} to Theorem \ref{thm.babyproj} via Balog--Szemeredi---Gowers.  
\begin{theorem}\label{thm.adultproj}
      Let $d$ be a positive integer. Let $E$ be a normed division algebra over $\Q_p$ or $\R.$ Let $0 < s \leq \sigma  \leq 2d, t > 0.$ There exists a universal $C \geq 1$ and $c_0 = c_0(s,\sigma,d,t) > 0$ so that the following holds for all $\de,\e >0$ small enough.
    
 Let $G \subset  B_{E \times E}(0,1)$ be a $(\de,s,\demm)$-set with $|G| = \de^{-\sigma}.$ Let $X \subset B_E(0,1)$ be a $(\de,t,\demm)$-set which avoids sub-algebras. Then there exists $F \subset G$ with $|F| > \de^{Cc_0} |G|$ and an $x \in X$ so that
        \begin{equation}\label{eq.conclusion}
        \cn{\pi_x(F')} > \de^{-c}|G|^{1/2} \text{ for all } F' \subset F \text{ with } |F'| > \depp |F|.
        \end{equation}
        
            Further, provided that $E$ is not a non-commutative division algebra of $\Q_p,$ then $c_0 = c_0(s,\sigma,t,t/d).$
        \end{theorem}

We make a few remarks about the dependence of $c$ and $c_0$ on other parameters. For normed division algebras over $\R$  we stress that our proofs output values of $c$ and $c_0$ that \textit{do} depend on $d.$ However by Wedderburn's theorem, all normed division algebras over $\R$ must be isomorphic to $\R, \C$ or $\mathbb{H},$ and so $d \leq 4.$ So by taking $c, c_0$ to be uniform over these three choices, we get a value independent of $d.$ A careful inspection of our proof of Theorem \ref{thm.babyproj} would show that the value of $c$ outputted for $\mathbb{H}$ would be smaller than that of (say) $\R,$ but merely from a qualitative point of view $c$ will not depend on $d.$

When the normed division algebra is over $\Q_p$ the story is different, as these may have arbitrarily large dimension. Here, provided that $E$ is not a non-commutative division algebra over $\Q_p,$ then the gain in exponent will depend on the normalised dimension of the set of projections (its usual dimension divided by $d$), rather than on $d$ itself. Our proof for the commutative case (which is based on the one in \cite{nic}) very much uses commutativity, and so does not seem to extend to the non-commutative case.

\hfill 
        
        Theorem \ref{thm.adultproj} is sharp in a number of ways. We remark that we are not able to upgrade Theorem \ref{thm.adultproj} to a result about `most' $x \in X$ with a statement in the spirit of \cite[Theorem 5]{bou} or \cite[Theorem 1.3]{o2025simple}. A counterexample in $\C$ would be (supposing that $\de^{-1}$ is an integer) 
        $$A = \{0,\de,2\de\dots, 1\},$$
        $$G = A \times A,$$
        $$X = A \cup \{i\}.$$
        Note that $A$ is a $(\de,1,C)$-set, and $G$ is a $(\de,2,C)$-set for some $C \sim 1.$ Then for all elements of $X$ apart from $i$ we have
        $\cn{\pi_x(G)} \sim |G|^{1/2}.$
        (Indeed $\cn{\pi_i(G)} \sim |G|.$)

We are also not able to upgrade Theorem \ref{thm.adultproj} to take $F = G,$ which is possible when $E = \R.$ See \cite[Theorem 1.3]{o2025simple}. A counterexample in $\C$ is a variant of the example above. Letting $A$ be as above, consider
 $$G_0 = A \times A,$$ 
 $$G_1 = iA \times iA,$$
$$G = G_0 \cup G_1$$
and 
$$X = A \cup iA.$$
Then $G$ is a $(\de,2,C)$-set and $X$ is a $(\de,1,C)$-set for some $C \sim 1.$ Note that
$$|G_0|, |G_1| \geq |G|/2 \gg \de^{\e}|G|,$$
provided that $\de$ and $\e$ are small enough.
Take $x \in X.$ If $x \in \R$ then
$$\cn{\pi_x (G_0)} \sim |G|^{1/2}.$$
If $x \in \Im $ then
$$\cn{\pi_x (G_1)} \sim |G|^{1/2}.$$
We remark here that for this example the conclusion of Theorem \ref{thm.adultproj} holds for each $x \in X,$ in the sense that for all $x \in X$ we may find $F_x \subset G$ with the desirable density so that \eqref{eq.conclusion} holds.

We are indeed able to upgrade our conclusion provided that $X$ strongly avoids sub-algebras. One should view Section \ref{subsect.defs} for this definition.

\begin{theorem}\label{thm.adultproj2}
      Let $d$ be a positive integer. Let $E$ be a normed division algebra over $\Q_p$ or $\R.$ Let $0 < s \leq \sigma  \leq 2d, t > 0.$ There exists $c_0 = c_0(s,\sigma,d,t) > 0$ so that the following holds for all $\de,\e >0$ small enough.
    
 Let $G \subset  B_{E \times E}(0,1)$ be a $(\de,s,\demm)$-set with $|G| = \de^{-\sigma}.$ Let $X \subset B_E(0,1)$ be a $(\de,t,\demm)$-set which strongly avoids sub-algebras. Then there exists $Y \subset X$ with $|Y| \geq (1-\dep)|X|$ so that
        $$\cn{\pi_x(F)} > \de^{-c}|G|^{1/2} \text{ for all } x \in Y \text{ and } F \subset G \text{ with } |F| > \depp |G|.$$
        
            Further, provided that $E$ is not a non-commutative division algebra of $\Q_p,$ then $c_0 = c_0(s,\sigma,t,t/d).$
        \end{theorem}

\begin{remark}
    We remark that we may consider the map $\pi_x$ defined by $(a,b) \mapsto a+bx$ with only superficial adjustments to the proofs. 
\end{remark}
\subsection*{Acknowledgements}
    The author is indebted to Pablo Shmerkin and Hong Wang. Thanks are extended to Joshua Zahl for all of his support.

\section{Preliminaries}
\subsection{Basic definitions}\label{subsect.defs}
 Throughout $E$ will be a normed division algebra 
 over $\Q_p$ or $\R.$ For two functions $f,g :(0,1] \rightarrow \R_{\geq 0} $ we write
 $$f \lesssim g \text{ if there exists } C > 0 \text{ such that } f \leq Cg. $$
 We write
 $$f \lessapprox g \text{ if for all $\ups > 0$ there exists } C > 0 \text{ such that } f \leq C\de^{-\ups}g. $$

 We say that a $A \subset E$ is a $(\de,s,C)$\textit{-set} if 
        $$\cn{A \cap B(x,r)} \leq Cr^s\cn A \text{ for all } x \in A, r \geq\de.$$ 
        Recall $N_\de(\ \cdot \ )$ denotes the covering number by balls of radius $\de.$ 

        An immediate corollary of this definition, which we shall use without mention, is that
$$\cns\rho A \geq \rho^{-s}/C$$ for all $\rho > \de.$ 
        We use $A_\de$ to denote the $\de$-neighbourhood of $A.$

        We use $\dy\de A$ to denote a partition of $A$ into cubes of side $\de.$

 We say that a $\de$-separated subset $A\subset E$ is \textit{uniform} if for all $\de < \rho$ and all $I,J \in \dy \rho A$ we have 
 $|A \cap I| \approx |A \cap J|.$
        
        Let $0 < \de < 1, 0 < s \leq d, C \geq 1.$ We say that a $(\de,s,C)$-set $A$ \textit{avoids sub-algebras} if for all sub-algebras $F \subset E$ there exists $a \in A$ so that
        \begin{equation}
            d(a,F) \geq 1/C.
        \end{equation}

        We say that a $(\de,s,C)$-set $A$ \textit{strongly avoids sub-algebras} if for all $B \subset A$ with $|B| \geq |A|/C,$ and for all sub-algebras $F \subset E$ there exists $b \in B$ so that
        \begin{equation}
            d(b,F) \geq 1/C.
        \end{equation}

        For an integer $N$ we write 
        $$NA = \overset{\times N}{A + \dots + A},$$
        and 
        $$A^{(N)} = \overset{\times N}{A \cdots A}.$$
 \subsection{Uniform subsets}
 We often use the Proposition below without reference, often replaced with the phrase `we find a dense uniform subset of -.'
    \begin{prop}\label{prop.uniform}
        Let $A \subset B(0,1)$ be $\de$-separated and let $\e > 0.$ Provided that $\de, \e >0$ are small enough, $A$ contains a uniform subset $A_0$ with $|A_0| > \de^\e|A|.$ In particular, if $A$ is a $(\de,s,\demm)$-set, then for all $\de < \rho < 1$ the set $A_0$ is a $(\rho,s,\dem)$-set.

    \end{prop}
        We give a proof of this result when $E$ is a $d$-dimensional division algebra over $\Q_p.$ When $E \in \{\R,\C, \mathbb{H}\},$ a reference would be \cite[Lemma 2.15] {OrponenShmerkin23b}, or simply replace $p$ with $2$ in the proof below. 
        \begin{proof}[Proof of Proposition \ref{prop.uniform}] Let $m$ and $T$ be integers selected later, so that (without loss of generality) $\de = p^{-mT}.$ We perform a by now standard up-the-tree refinement. 

Select $A_m = A.$ Suppose that sets $A_m \supset A_{m-1} \supset \ldots \supset A_{m-j}$ have been selected so that for all $0 \leq k \leq j$ and $I,J \subset \dy{p^{-(m-k)}}{B(0,1)}$ 
 we have
$$p^{-1}|A_{m-j} \cap I| \leq |A_{m-j} \cap J| \leq p|A_{m-j} \cap I|.$$ 
We select $A_{m-j-1}$ as follows: $A_{m-j}$ is covered by the intervals contained in $\dy{p^{-(m-j-1))}}{B(0,1)}.$
Consider those sets with non-empty intersection with $A_{m-j},$ and call this collection $\cU.$ The number of elements $\dy{p^{-(m-j)}} {A_{m-j}}$ in each such interval is between $1$ and $p^{dT}.$ By pigeonholing, we may select an integer $1 \leq i_0 \leq dT+1$ and a $\cV \subset \cU$ so that for all $V \in \cV$ we have
$$ p^{i_0-1} \leq |A_{m-j} \cap V| < p^{i_0},$$
and
$$|A_{m-j} \cap \cup\cV| \geq \tfrac{1}{dT+1}|A_{m-j}|.$$
Select $A_{m-j-1} = A_{m-j} \cap \cup \cV.$ Once $A_0$ is selected we note that
\begin{equation}
    |A_0| \geq (dT+1)^{-m}|A|. 
\end{equation}
Provided that $m$ and $T$ are chosen appropriately large it will be the case that
$$\frac{1}{(dT+1)^m} \geq p^{-\e mT} = \depp.$$ The uniformity of $A_0$ is readily checked. 
        \end{proof}
\subsection{Escape from subspaces}
We recall \cite[Proposition 4]{nic}. 
    \begin{theorem}\label{thm.escape}
        Let $E$ be a normed division algebra of dimension $d$ over $\R$ or $\Q_p.$ Let $A \subset B(0,1)$ be such that for all $F \subset E$ there exists $a \in A$ so that $d(a,F) \geq \rho.$ Then there exists $v_1,\ldots,v_d \in A^{(d)}$ so that
        $$\{v_1,\dots,v_d\}$$
        forms an almost-orthonormal basis of $E,$ in the sense that 
        $$\det(v_1,\ldots,v_d) \geq \rho^{O(1)}.$$
    \end{theorem}
\section{Expansion under addition and multiplication}
\subsection{Parameters}
 For parameters $0 < s < d$ let $c_1 = c_1(s,d)$ be chosen as
\begin{equation}\label{eq.cone}
    c_1 =\frac{s(1-s/d)}{4}
\end{equation}
Since $s$ will represent the `dimension' of our discretised set, we use $s' = s/d$ to denote the normalised dimension with respect to the dimension of $E$ over its base field. 
\subsection{Non-field case}
We first prove a general expansion result where the number of products of $A$ needed to experience growth is grows linearly with $d.$ The dependence on $d$ for the number of products may be explained by an application of Theorem \ref{thm.escape}. 
\begin{theorem}\label{thm.expand}
    Let $E$ be a $d$-dimensional normed division algebra over $\Q_p$ or $\R.$ Let $0 < s < d.$ The following holds for all $\de,\e >0$ small enough. Let $A \subset B(0,1)$ be a $\de$-separated $(\de,s,\demm)$-set that avoids sub-algebras. Then
    \begin{equation}
        \cn{10(A^{(2d)}-A^{(2d)})} \geq \de^{-s-c_1}.
    \end{equation}
\end{theorem}
We follow the strategy of \cite{gut} which is repeated in \cite{nic}.
\begin{proof}[Proof of Theorem \ref{thm.expand}]
\begin{steps}

\step{Pre-processing} We assume that $A$ is uniform. To avoid division by elements close to $0$ we consider $A_0 = A \setminus B(0,\de^{2(\e/s)}).$ Since
\begin{equation}
    |A \cap B(0,\de^{2(\e/s)})| \leq \de^{\e}|A| \ll |A|,
\end{equation}
the set $A_0$ remains a $(\de,s,\de^{-\e})$-set, provided $\de,\e$ are small enough. Re-label $A_0$ by $A.$

\step{Setup} Let $\de < \rho < \de^{1/3}$ be a parameter selected later. Select $\De = \de/\rho^3.$ Set 
\begin{equation}
    Q_1 = \{(a-b)(c-d)^{-1} : a,b,c,d \in A \text{ and }|c-d| > \rho\},
\end{equation}
\begin{equation}
    Q_2 = \{(c-d)^{-1}(a-b) : a,b,c,d \in A \text{ and }|c-d| > \rho\},
\end{equation}
and 
$$Q_0 = Q_1 \cup Q_2.$$
Without a total loss of generality, let us suppose that $|Q_1| \geq |Q_0|/2;$ set $Q$ to be a maximal $\De$-separated subset of $Q_1.$ Apply Theorem \ref{thm.escape} to find $v_1,\dots, v_d \in A^{(d)}$ with $\det(v_1,\dots,v_d) > \dep,$ so that $\{v_1,\dots,v_d\}$ forms a basis of $E.$ 
\step{Dense case}
The proof when $E$ is over $\Q_p$ or $\R$ now diverge. First we consider the case when $E$ is over $\R.$ For each $\vi = (i_1,\dots,i_d) \in \{0,1\}^d$ consider the map $f_\vi : E \rightarrow E$ defined by
\begin{equation}
    f_\vi(x) = f_\vi\left(\sum_{j=1}^d x_j v_j\right) =\sum_{j=1}^d \left(\frac{x_j+i_j}{2}\right) v_j
\end{equation}

Suppose first that for all $x \in Q$ and each $\vi \in \{0,1\}^d$ we have
\begin{equation}\label{eq.assumptiononQ}
d(f_\vi (x), Q) \leq \De.
\end{equation} 
The aim is to show that 
$$v_1[0,1) + \cdots + v_d[0,1)\subset Q_\De.$$ 
This will mean that 
\begin{equation}\label{eq.coveringQ}
\cn{Q_\De} \geq \dep \De^{-d}.
\end{equation}
We will do this by showing by induction that all numbers $\alp$ of the form
$$\alpha = \sum_{j=1}^d v_j k_j2^{-n}$$
for $0 \leq k_j < 2^{n}$ are contained in $Q_\De.$
Since $0 \in Q,$ the case $n=0$ follows after applying each $f_\vi$ to $0.$ Now fix 
$$\beta = \sum_{j=1}^d v_j k_j 2^{-n-1}$$
for $0 \leq k_j < 2^{n+1}.$ If $k_j < 2^n$ select $i_j = 0$ and $l_j = k_j.$ Otherwise select $i_j = 1$ and $l_j = k_j - 2^n.$ By the inductive hypothesis, 
$$\gamma  = \sum_{j=1}^dv_jl_j2^{-n} \in Q_\De,$$
and by design, 
$$f_\vi(\gamma) = \beta.$$
Therefore by \eqref{eq.assumptiononQ} $\beta \in Q_\De.$

Now let $$f: A \times Q \rightarrow A(A-A)\times A(A-A)$$ be defined by
$$f(a,x) = (as_x, at_x),$$
where for each $x$ we choose some representatives of our choice $s_x \in A-A$ and $t_x \in (A-A) \setminus B(0,\rho)$ so that $x = s_x t_x^{-1}.$

We now show that if we know a point in the image up to some error of $\sim\de,$ then we may locate its pre-image up to an error of $\sim \de/\rho.$ Let $(y,y')$ be in the image, and let 
$$(z,z') = (y+ O(\de), y' + O(\de)).$$ 
There exist $a \in A$ and $x = s_xt_x^{-1}$, with $ |t_x| \geq \rho$  so that $(y,y') = (as_x,at_x).$ 
Since $|t_x| \geq \rho$ we have
\begin{align}
zz'^{-1} &= (y+O(\de))(y' + O(\de))^{-1}\\
&= s_xt_x^{-1} + O(\de/\rho) \\
&= x + O(\de/\rho).
\end{align}
Once $x$ is determined, we know our choices $s_x$ and $t_x,$ therefore $$at_x = y' + O(\de),$$ and so $$a = y't_x^{-1} + O(\de/\rho).$$ Since $\De > \de/\rho$ using \eqref{eq.coveringQ} we therefore have
\begin{equation}
    \cn{A(A-A)}^2 \gtrsim \cns{\de/\rho}A\cns \De Q \geq \dep (\rho/\de)^s\De^{-d} = \dep\rho^{s+3d} \de^{-s-d},
\end{equation}
and so
\begin{equation}\label{eq.comparing}
    \cn{A(A-A)} \geq \dep \rho^{\frac{s+3d}{2}}{\de^{-\frac{s+d}{2}}}.
\end{equation}
\step{Sparse case}
Now suppose that there exists an $x \in Q$ and an $\vi$ so that $$d(f_\vi(x),Q) > \De.$$ Set $y = f_\vi(x),$ and let $s,t,u,v \in A$ with $|u-v| \geq \rho$ be so that $x = (s-t)(u-v)^{-1}.$ We perform the argument in the most superficially complicated case when $\vi = (1,\ldots,1).$ We have
\begin{align}
    y &=  \frac 1 2 \left(x + \sum_{j=1}^d v_j\right)\\
    &= \frac 1 2 \left((s-t) + \sum_{j=1}^dv_j(u-v)\right)(u-v)^{-1}\\
    &= \left((s-t) + \sum_{j=1}^dv_j(u-v)\right)(u-v+u-v)^{-1}.
\end{align}
We may therefore write $y = pq^{-1},$ where 
$$p \in A-A + dA^{(d)}(A-A),$$
and 
$$q \in (A-A)+(A-A),$$
with
$|q| \geq \rho.$
 The aim is now find an upper bound for the number of quadruples
\begin{equation}
 Y = \{(a_1,a_2,a_3,a_4) \in A^4 : |a_1q + a_3p - (a_2q + a_4p)| \leq \de\}. 
\end{equation}
For $(a_1,a_2,a_3,a_4) \in Y$ we must have
\begin{equation}
    |(a_1-a_2)(a_3-a_4)^{-1} - pq^{-1}| \leq \de|a_3-a_4|^{-1}|q|^{-1} \leq |a_3-a_4|^{-1}\de/\rho. 
\end{equation}
Since $d(pq^{-1},Q) > \De,$ we have $|a_3-a_4| \leq \rho^2.$ Fixing $a_3,$ by non-concentration, this gives us $\demm \rho^{2s}|A|$ choices for $a_4.$ Once $a_2,a_3,a_4$ are fixed then
\begin{equation}
    |a_1 - (a_2q-a_3p +a_4p)q^{-1}| \leq \de|q|^{-1} \leq \de/\rho,
\end{equation}
giving us $\demm (\de/\rho)^s|A|$ choices for $a_1.$ In summary,
\begin{equation}
    |Y| \leq \dem (\de/\rho)^s\rho^{2s}|A|^4 = \dem \de^s \rho^s |A|^4.
\end{equation}
An application of Cauchy--Schwarz gives
\begin{equation}
    \cn{(A-A)A + dA^{(d)}(A-A)A + (A-A)A + (A-A)A} \geq \cn{pA + qA} \geq \dep \rho^{-s}\de^{-s}.
\end{equation}
Now set $\rho = \de^{\frac{d-s}{3(d+s)}}.$ Comparing with \eqref{eq.comparing} we have,
\begin{equation}
    \cn{2d(A^{(2d)} - A^{(2d)})} \geq \dep \de^{-s\left(1+\frac{d-s}{3{d+s}}\right)} \geq \de^{-s-c_1}.
\end{equation}
\step{Modifications for $\Q_p$}
When $E$ is over $\Q_p$ we explain the necessary modifications. The proof here is slightly simpler and the number of sums of $A$ needed in the end is $O(1)$ and not $O(d).$ 

Recall that $\{v_1,\dots,v_d\}$ forms a basis of $E$ with 
\begin{equation}\label{eq.determinant}
\det(v_1,\dots,v_d) > \dep.
\end{equation}For each $1 \leq j \leq d$ consider the map $f_j : E \rightarrow E$ defined by
\begin{equation}
f_j(x) = x + v_j.
\end{equation}
Suppose first that for all $x \in Q$ and each $1 \leq j \leq d$ we have 
\begin{equation}\label{eq.contradiction1}
d(f_j(x), Q) \leq \De.\end{equation}
This will mean that 
\begin{equation}\label{eq.contradiction}
f_j(Q_\De) \subset Q_\De.
\end{equation}
To see this, take $y \in Q_\De,$ there exists $x \in Q$ with $d(x,y) \leq \De.$ By \eqref{eq.contradiction1} we find $z \in Q$ so that $d(x+v_j,z) \leq \De.$ Therefore, using the ultra-metric property, we have
$$d(y+v_j,z) \leq \max \{d(y+v_j,x+v_j), d(x+v_j,z)\} \leq \De,$$
verifying \eqref{eq.contradiction}.

We then claim that 
\begin{equation}
    v_1\Z_p + \cdots + v_d \Z_p \subset Q_\De.
\end{equation}
Suppose not. Since $0 \in Q_\De$ by \eqref{eq.contradiction} we have $v_1 + \dots + v_d \in Q_\De.$ By considering the lexicographical order, we may find $j$ and $n_1,\dots,n_d$ so that 
$$v_1n_1 + \dots + v_dn_d \in Q_\De,$$
but 
$$v_1n_1 + \dots + v_j(n_j+1) + \dots + v_dn_d \notin Q_\De.$$
This contradicts \eqref{eq.contradiction}. Since $\cns \De {\Z_p^d} = \De^{-d}$ by \eqref{eq.determinant} we have 
\begin{equation}
    \cns \De Q \geq \dep\De^{-d}. 
\end{equation}
We are then able to proceed similarly to before.

If there exists $x \in Q$ and $1 \leq j \leq d$ so that 
$$d(x + v_j, Q) \geq \De,$$
then we may write $x + v_j = p q^{-1}$ where
$p \in A-A + A^{(d)}(A-A)$
and 
$q \in (A-A) \setminus B(0,\rho).$
We then proceed as before. It is this part that guarantees the desirable $O(1)$ sums instead of $O(d).$
\end{steps}
\end{proof}
  \subsection{Field case}
We now prove an expansion result when $E$ is a field. The number of sums and products needed is independent of $d.$ Note that this statement is proved in Theorem \ref{thm.expand} for $\R$ and $\C$ (and $\mathbb{H}).$ The proof below could be modified to work for $\R$ and $\C$ should that be desirable to the reader.
\begin{theorem}\label{thm.expand2}
    Let $d$ be a positive integer. Let $E$ be a commutative $d$-dimensional extension of $\Q_p.$ Let $0 < s < d.$ The following holds for all $\de,\e >0$ small enough. Let $A \subset B(0,1)$ be a $\de$-separated $(\de,s,\demm)$-set that avoids sub-algebras. Then
    \begin{equation}
        \cn{10(A^{(10)}-A^{(10)})} \geq \de^{-s-c_1}.
    \end{equation}
    \end{theorem}
    \begin{proof}
        \begin{steps}

\step{Pre-processing} As in the proof of Theorem \ref{thm.expand}.
\step{Setup}
Let $\de < \rho < \de^{1/3}$ be a parameter selected later. Select $\De = \de/\rho^3.$ Set 
\begin{equation}
    Q= \left\{\frac{a-b}{c-d} : a,b,c,d \in A \text{ and }|c-d| > \rho\right\}.
\end{equation}
\step{Dense case}
Suppose that for all $x,y \in Q$ we have
\begin{equation}
    d(x+y,Q) \leq \De \text{ and } d(xy,Q) \leq \De.
\end{equation}
We will use the ultra-metric property of the metric on $E$ to show that 
\begin{equation}\label{eq.add}
Q_\De + Q_\De \subset Q_\De
\end{equation}
and 
\begin{equation}\label{eq.mult}
Q_\De Q_\De \subset Q_\De.
\end{equation}

Let $x,y \in Q_\De.$ Then there are $z,w \in Q$ so that 
$$d(x,z) \leq \De \text{ and } d(y,w) \leq \De.$$
Using the ultra-metric property, we have
\begin{equation}
    d(x+y,Q) \leq \max \{d(x+y,x+w), d(x+w,z+w), d(z+w,Q)\} \leq \De. 
\end{equation}
Similarly, 
\begin{equation}\label{eq.closure}
    d(xy,Q) \leq \max\{ d(xy,xw), d(xw,zw), d(zw,Q)\} \leq \De,
\end{equation}
using the fact that $A \subset B(0,1).$

Since $0,1 \in Q,$ stability under addition means that $\Z_p \subset Q_\De.$

We now show that $Q_\De$ is not trapped in any proper sub-algebra. For any interval $I$ of length $ \de^{2(\e/s)}$ by the non-concentration of $A$ we have
$$|A \cap I| \leq \de^\e|A|,$$
and so we have
$$\cns {\de^{2(\e/s)}} {A} \geq \demm.$$
We may therefore find $a,b \in A$ with \begin{equation}\label{eq.choices}|a-b| \geq \de^{2(\e/s)}.\end{equation} Consider
$$\frac{A-A}{a-b},$$
and let $F \subset E$ be a proper sub-algebra. Then $(a-b)F$ is also a proper sub-algebra. Since $A$ avoids sub-algebras we may find $c \in A$ so that $d(c,(a-b)F) \geq \depp.$ Let $d$ be another element in $a$ with $|d| \leq d(c,(a-b)F).$ Let $f \in F$ be such that $d(c,(a-b)f) = d(c,(a-b)F).$ Then using the ultra-metric property, we have
 $$d(c-d,(a-b)f) = \max \{d(c-d,(a-b)f-d), d((a-b)f,(a-b)f-d)\} = d(c,(a-b)f) \geq \de^\e. $$
 Therefore,
$$d((c-d)(a-b)^{-1},F) = |a-b|^{-1} d((c-d), (a-b)F) \geq \min\{1,\de^{\e (1-2/s)}\} \geq \depp.$$
We may therefore apply Theorem \ref{thm.escape} to $Q$ to find $v_1,...,v_d \in Q^{(d)}$ so that that
$$\cns\De {v_1Q + \dots + v_d Q }\gtrsim \dep \De^{-d}.$$ Since $Q_\De$ is closed under addition and multiplication we have
\begin{equation}
    \cns\De {Q }\gtrsim \dep \De^{-d}.
\end{equation}
Repeating the argument as in the proof of Theorem \ref{thm.expand}, we obtain that 
\begin{equation}
    \cn{A(A-A)} \geq \dep \rho^{\frac{s+3d}{2}}{\de^{-\frac{s+d}{2}}}.
\end{equation}
\step{Sparse case} Now suppose that there exist $x,y \in Q$ so that either 
$$d(x+y,Q) > \De \text{ or } d(xy,Q) > \De.$$ 
If the former occurs, write
$$x+y = \frac u v$$
where $$u \in 2(A-A)(A-A) \text{ and } v \in (A-A)(A-A).$$ If the latter occurs, write
$$xy = \frac u v, $$
where $$u,v \in (A-A)(A-A).$$
Similarly to the sparse case in the proof of Theorem \ref{thm.expand} we obtain
\begin{equation}
   \cn{3(A-A)(A-A)A} \geq \cn{uA + vA} \gtrsim \dep \rho^{-s}\de^{-s}.
\end{equation}
Choosing $\rho = \de^{\frac{d-s}{3(d+s)}}$ we obtain in all cases.
\begin{equation}
    \cn{10(A^{(10)} - A^{(10)})} \geq \dep \de^{-s\left(1+\frac{d-s}{3{d+s}}\right)} > \de^{-s-c_1}.
\end{equation}
\end{steps}
    \end{proof}

\subsection{An iterative expansion estimate}
    
\begin{corollary}\label{cor.expand}
      Let $E$ be a $d$-dimensional normed division algebra over $\Q_p$ or $\R.$ Let $0 < s < d.$ The following holds for all $\de,\e >0$ small enough. Let $A \subset B(0,1)$ be a $(\de,s,\demm)$-set which avoids sub-algebras. Then $$20(A^{(2d)} - A^{(2d)}) \cap B(0,1)$$ contains a $(\de,s+c_1/2,\demm)$-set.
      Provided that $E$ is not a non-commutative division algebra over $\Q_p,$ then we may replace $2d$ with $10.$
      \end{corollary}
      We give a proof in general using Theorem \ref{thm.expand}. It is easy to see that $2d$ may be replaced with $10$ if we use Theorem \ref{thm.expand2} instead.
        \begin{proof}
            Replace $A$ with a dense uniform subset. Then for each $\de < \rho < 1$ the set $A$ is a $(\rho,s,\dem)$-set. Apply Theorem \ref{thm.expand} at scale $\rho$ to find that 
            \begin{equation}
                \cns{\rho}{10(A^{(2d)}-A^{(2d)})} \geq \rho^{-s-c_1}.
            \end{equation}
            Therefore, there is a ball $B$ of radius $1/2$ so that
            \begin{equation}
                \cns{\rho}{10(A^{(2d)}-A^{(2d)})\cap B} \geq \rho^{-s-c_1/2}.
            \end{equation}
        We may find $x \in 10(A^{(2d)} - A^{(2d)})$ so that
        $$10(A^{(2d)} - A^{(2d)}) - x \subset B(0,1).$$
            Replacing
            $$10(A^{(2d)}-2d^{(2d)})$$
            with 
            $$10(A^{(2d)}-A^{(2d)}) - 10(A^{(2d)}-A^{(2d)}) \subset 20(A^{(2d)}-A^{(2d)})$$
            therefore ensures that 
            \begin{equation}
                \cns{\rho}{ 20(A^{(2d)}-A^{(2d)})\cap B(0,1)} \geq \rho^{-s-c_1/2}.
            \end{equation}
            Now replace $20(A^{(2d)}-A^{(2d)}) \cap B(0,1)$ with dense uniform subset $X.$ It follows that for each ball $B$ of radius $\rho$ which intersects $A$ we have
    \begin{equation}
        |X \cap B| \leq |X|\cns \rho X ^{-1} \leq \demm \rho^{s+c_1/2}|X|,
    \end{equation}
        as required.
        \end{proof}
        Recall for $s,t$ we denote $s' = s/d, t' = t/d.$
        \begin{prop}\label{prop.expansion}
              Let $0 < s < t < d.$ Let $A \subset B(0,1)$ be a $(\de,s,\demm)$-set that avoids sub-algebras. Then there exists a positive integer $N = N(s,t,d)$ so that 
            $$N(A^{(N)} - A^{(N)}) \cap B(0,1)$$
            contains a $(\de,t,\demm)$-set. Further, provided that $E$ is not a non-commutative extension of $\Q_p,$ then $N$ will only depend on $s,s,t,t'.$  
\begin{proof}
    Apply Corollary \ref{cor.expand} to $A.$ Let $M_0,N_0$ be the positive integers given to us so that
    $$M_0(A^{(N_0)} - A^{(N_0)}) \cap B(0,1)$$
    contains a $(\de,s+c_1/2,\demm)$-set. It is important to note that $M_0$ and $N_0$ do not depend on $s.$ If $s+c_1/2 \geq t$ then we are of course done. Otherwise, we keep applying Corollary \ref{cor.expand}, where after $n$ iterations, the set outputted will be a $(\de,s_n,\demm)$-set, where $s_n$ is given by the recursive formulae:
    $$s_0 = s;$$ 
    $$s_n = s_{n-1} + c_1(s_{n-1})/2.$$

    Suppose first that $t < d/2.$ Since
    $$c_1(s_n) > c_1(s_{n-1}) > \dots > c_1(s), $$
    for the number $n$ so that $s_n \geq t,$ we have   
    $$n \sim \frac{t-s}{c_1(s)}.$$
When $t \geq d/2$ we iterate 
$$\sim \frac{d/2-s}{c_1(s)}$$
times until the output contains a $(\de,d/2,\demm)$-set. After which, $c_1$ is a decreasing function, but still bounded below by $c_1(t).$ We therefore need another 
$$\sim \frac{t-d/2}{c_1(t)}$$
iterations so the output contains a $(\de,t,\demm)$-set. In either case, we may take 
\begin{equation}\label{eq.numberofiterations}
n = n(s,s',t,t') \sim \max \left\{ \frac{t-s}{c_1(s)}, \frac{t-s}{c_1(t)}\right\}.
\end{equation}

Now set $B= A^{(N_0)}.$ Note that 
$$(M_0(B-B))^{(N_0)} \subset M_0^{(N_0)}(B-B)^{(N_0)} \subset M_0^{(N_0)}(B^{(N_0)} - B^{(N_0)}) = M_0^{(N_0)}(A^{(N_0)^2} - A^{(N_0)^2}).$$
Therefore setting $$C = M_0(A^{(N_0)} - A^{(N_0)}),$$
we have
$$M_0(C^{(N_0)} - C^{(N_0)}) \subset 2M_0  M_0^{N_0}(A^{(N_0)^2} - A^{(N_0)^2}) \subset M_0^{2N_0}(A^{(N_0)^2} - A^{(N_0)^2}).$$
Iterating $n$ times we obtain that 
$$M_0^{2^nN_0^n}(A^{(N_0^n)}- A^{(N_0^n)}) \cap B(0,1)$$
contains a $(\de,t,\demm)$-set. Provided that $E$ is not a non-commutative extension of $\Q_p$ then from Corollary \ref{cor.expand} we know that $M_0, N_0 \sim 1,$ in particular we may take
$$N = 20^{20^n},$$
in light of \eqref{eq.numberofiterations}. Otherwise, we may take 
$$N = 20^{(4d)^n}.$$
\end{proof}
        \end{prop}

    \section{Projections of Cartesian products}
    Below is a variant of \cite[Exercise 2.8.4]{tao} adapted to the discretised setting.

\begin{lemma}\label{lem.tv}
    Let $0 < s \leq \sigma < t \leq d.$ There exists $c = c(s,t,\sigma) > 0$ so that the following holds for all $\de,\e >0$ small enough.
    
    Let $A \subset B(0,1)$ be a $(\de,s,\demm)$-set with $|A| = \de^{-\sigma}.$ Let $X \subset B(0,1)$ be a $(\de,t,\demm)$-set. There exists an $x \in X$ so that
    $\cn{A + xA} > \de^{-c}|A|.$
    \begin{proof}
 We aim to bound the cardinality of the the set
        \begin{equation}
            Q = \{(a,b,c,d,x) \in A^4\times X : |a+xb - (c-xd)| \leq \de \}.
        \end{equation}
        Let $\rho \geq \de$ be selected later. Suppose that the quintuples with $|b-d| \leq \rho$ make up at least half of all quintuples. Fix $a,b \in A$ and $x \in X.$ By the non-concentration of $A,$ the number of admissible $d$ is at most $\de^{-\e}\rho^s|A|.$ Now $a,b,d,x$ are fixed, and $c$ must now satisfy 
        $$|c - (a+ xb - xd| \leq \de,$$
        that is, there is at most one choice for $c.$ 
        Therefore, in this case we have,
        \begin{equation}
            |Q| \lesssim \de^{-\e}\rho^s |A|^3|X|. 
        \end{equation}
        
        In the other case, when at least half of the quintuples of $Q$ satisfy $|b-d| \geq \rho,$ we have
        \begin{equation}
            |Q| \lesssim |\{(a,b,c,d,x) \in A^4\times X : |x - (a-c)(b-d)^{-1}| \leq \de/\rho \}|.
        \end{equation}
Fix $a,b,d \in A$ with $|b-d| \geq \rho,$ of which there are at most $|A|^3$ choices. Since $x \in B(0,1),$ we must have $|a-c| \leq \rho,$ so there are $\demm\rho^s |A|$ choices for $c.$ With $a,b,c,d$ selected, this leaves $\demm (\de/\rho)^t |X|$ choices for $x.$ Therefore
\begin{equation}
    |Q| \lesssim \de^{-2\e}\de^t \rho^{s-t}|A|^4|X|
\end{equation}
Now select $\rho = \de^{\frac{t-\sigma + \e}{t}}.$ We then have
\begin{align}
        |Q| &\lesssim \demm\de^{\frac{s(t-\sigma + \e)}{t}}|A|^3|X| + \de^{-2\e}\de^{t+\frac{(s-t)(t-\sigma + \e)}{t}}|A|^4|X| \\ 
    & = \demm\de^{\frac{s(t-\sigma + \e)}{t}}|A|^3|X| + \de^{-2\e}\de^{t- (t-\sigma + \e) +\frac{s(t-\sigma + \e)}{t}}|A|^4|X|\\
    &=  \demm\de^{\frac{s(t-\sigma + \e)}{t}}|A|^3|X|.
\end{align}
Set $c = \frac s t (t-\sigma + \e) > 0$ provided that $\e$ is small enough. Therefore there exists an $x \in X$ so that
\begin{align}
    |\{(a,b,c,d) \in A^4 : |a+xb - (c-xd)| \leq \de \}| &\lesssim \de^c |A|^3.
\end{align}
A simple argument using Cauchy--Schwarz completes the proof. 
    \end{proof}
\end{lemma}
We now prove Theorem \ref{thm.babyproj}, where we follow the strategy of the proof of \cite[Theorem 1.2]{o2025simple}.
\begin{proof}[Proof of Theorem \ref{thm.babyproj}]
     
    Consider $X_0 = X \setminus B(0,\de^{2(\e/t)}).$ Since
\begin{equation}
    |X \cap B(\de^{2(\e/t)})| \leq \de^{\e}|X| \ll |X|,
\end{equation}
the set $X_0$ remains a $(\de,t,\de^{-\e})$-set, provided $\de,\e$ are small enough. Replace $X$ with $X_0.$ When $t > \sigma$ the result follows from Lemma \ref{lem.tv}. Otherwise, let $$\sigma < u < \sigma +\e$$ and apply Proposition \ref{prop.expansion} to find a positive integer $N$ and 
    \begin{equation}
        Y \subset (NX^{(N)} - NX^{(N)})\cap  B(0,1)
    \end{equation}
    which is a $(\de,u,\demm)$-set. Now apply Lemma \ref{lem.tv} to $A$ and $Y$ where we find $y \in Y$ so that
    \begin{equation}
        \cn {A+yA} > \de^{-c}|A|,
    \end{equation}
     Here $c$ is chosen so that $c = c(s,\sigma) > 0$ provided that $\e > 0$ is small enough.
    
    We now use Ruzsa calculus to complete the proof. We may write $y = y_1 - y_2$ for $y_1,y_2 \in NX^{(N)}.$ By the Pl\"unnecke--Ruzsa inequalities, and the triangle inequality, we have
    \begin{align}
        \de^{-c} |A| &< \cn{A+y_1A - y_2A}\\
        &\lesssim \cn{A+A}\cn{A+y_1A}\cn{A-y_2A}|A|^{-2}\\
        &\lesssim\cn{A+y_1A}^3\cn{A+y_2A}^3\cn{y_1A}^{-1}\cn{y_2A}^{-2}|A|^{-2}\\
        &\lesssim \de^{-O_N(\e)} \cn{A+y_1A}^3\cn{A+y_2A}^3|A|^{-5}
    \end{align}
    so, say for $y_1,$ we have
    \begin{equation}
        \cn{A+y_1A} \gtrsim \de^{-c/10}|A|.
    \end{equation}
We may write $y_1 = z_1 + \ldots + z_N$ for $z_j \in X^{(N)}.$ Again, by the Pl\"unnecke--Rusza inequalities, we have
\begin{align}
    \de^{-c/10}|A| &\les \cn{A+A}\prod_{j=1}^N\cn {A+z_jA}|A|^{-N}\\
    &\lesssim \dem \cn{A+z_1A}^3\prod_{j=1}^N\cn {A+z_jA}|A|^{-N-1}
\end{align}
and so there exists $z_j$ so that
    \begin{equation}
        \de^{-\frac{c}{100N}}|A| \lesssim \cn{A+z_jA}.
    \end{equation}
    We may write $z_j = w_1\cdots w_N.$ By successive uses of the triangle inequality we have
    \begin{align}
    \cn{A+w_1\ldots w_NA} &\lesssim \cn{A+w_1A}\cn{w_1A + w_1\ldots w_NA}\cn{w_NA}^{-1} \\
    &\lesssim \dem \cn{A+w_1A}\cn{A + w_2\ldots w_NA}|A|^{-1} \\
    &\lesssim \cdots\\
    &\lesssim \dem \prod_j \cn{A+w_jA} |A|^{-N+1}.
    \end{align}
    Therefore there exists $w_j \in X$ so that
    \begin{equation}
        \cn{A+w_jA} \gtrsim \de^{-\frac{c}{1000N^2}}|A|.
    \end{equation}
Completing the proof. The correct dependence on parameters will come from observing the increase depends on the parameter $N$ from Proposition \ref{prop.expansion} and the parameter $c$ from Lemma \ref{lem.tv}.
\end{proof}
\section{Projections of general sets}

We now prove the main result of the paper, Theorem \ref{thm.adultproj}.
\begin{proof}[Proof of Theorem \ref{thm.adultproj}]
\begin{steps}
\step{Initial reductions}
    We replace $X$ with $X \setminus B(0,\de^{\e/2t})$ as in previous proofs. 
    Suppose the result is false. We claim that we may find  $x_1,x_2,x_3 \in X$ and $G_3 \subset G_2 \subset G_1 \subset G $ so that 
    \begin{enumerate}
        \item $|G_{j+1}| > \depp |G_j|,$
        \item $|x_i - x_j| > \de^{\e/2t},$
        \item $\cn{\pi_{x_j}(G_{j})} \leq \de^{-c_0}|G|^{1/2}.$
    \end{enumerate}
Since the result is false, we may find $x_1 \in X$ and  $G_1 \subset G$ so that $(1)$ and $(3)$ hold. Since 
$$|X \cap B(x_1, \de^{\e/2t})| \leq \de^\e|X| \ll |X|,$$
we use the falseness of the result again, to find $x_2 \in X$ with $|x_1 - x_2| > \de^{\e/2t}$ and a $G_2 \subset G_1$ which satisfy (1) and (3). Iterating this step a final time proves the claim.

We identify $X$ in projective space via the identification of $x$ with the line spanned by $\{(1,x)\},$ thus identifying $\pi_x$ with the orthogonal projection to said line. This will affect metric quantities by an acceptable factor of $\dep.$ 
    \step{Reducing product-set structure}
   Suppose it were the case that $\pi_{x_1}$ is the projection to the abscissa and $\pi_{x_2}$ is the projection to the ordinate. This may be obtained by a simple coordinate change which we will explain at the conclusion of the proof. Set $A_1$ and $A_2$ to be maximal $\de$-separated subsets of the orthogonal projection of $G$ to the abscissa and the ordinate respectively. After identifying $A_1,A_2 \subset E,$ 
   we have 
    $$G \subset A_1 \times A_2.$$

    By property (3) above, we have 
    $$|A_1||A_2| \leq \de^{-2c_0-O(\e)}|G| = \de^{-2c_0 - O(\e)} |H|,$$
    and so 
    $$|G_3| \geq \de^{5c_0}|A_1||A_2|. $$
Also by $(1)$ and $(3)$, $$\de^{3\e}|G| \leq |G_3| \leq |A_1||A_2| \leq \de^{-c_0-O(\e)}|A_1||G|^{1/2}.$$
Thus
$|A_1| \geq \de^{5c_0}|G|^{1/2},$ and likewise the same bound holds for $|A_2|.$

\step{Applying Balog--Szemer\'edi--Gowers}
    Let $B_1$ be a maximal $\de$-separated subset of $x_3A_2.$ Set $H$ to be a maximal $\de$-separated subset of 
    $$\{(g,x_3g') : (g,g') \in G_3\}.$$
    Since we have
    $$\cn{A_1 \overset{H}{+} B_1} \leq \de^{-c_0}|G|^{1/2},$$
    provided that $\e >0$ is small enough, we apply Balog--Szemer\'edi--Gowers \cite[Exercise 6.4.10]{tao} to find $A' \subset A_1$ and $B \subset B_1$ with $|A'| \gtrsim \de^{c_1} |A_1|$ and $|B| \gtrsim  \de^{c_1} |B_1|$ so that
        \begin{equation}\label{eq.unaffected}
            \cn {A' + B} \lesssim \de^{-c_1}|A_1|^{1/2}|B_1|^{1/2}.
        \end{equation}
        and with
        $$|H \cap (A' \times B)| \gtrsim \de^{c_1}|A'||B|.$$
        Here $c_1 = O(c_0),$ with the implicit constant generic. 

        \step{Reducing $A'$ to $L^1$ support}
Let $\pi$ denote the projection to the abscissa. Set 
\begin{equation}\label{eq.bigfibres}
\cA = \{I \in \dy\de {A'} : |\pi^{-1}(I)\cap G| \geq \de^{10c_1}|G|^{1/2}\},
\end{equation}
and $A$ to be a maximal $\de$-separated subset of $\cup \cA \cap A'.$
Note that
$$|H \cap (A'\setminus A \times B)| \leq |A'|\de^{10c_1}|G|^{1/2} \lesssim \de^{5c_1}|A'||A_2| \lesssim \de^{4c_1}|A'||B| \ll |H \cap (A' \times B)|.$$
Therefore, 
\begin{equation}\label{eq.bigH}
|H \cap (A \times B)| \gtrsim \de^{c_1}|A'||B| \gtrsim \de^{3c_1}|A_1||B_1| > \de^{5c_1}|A_1||B_1|.
\end{equation}
In particular,
\begin{equation}\label{eq.boundonA}
|A| \gtrsim \de^{c_1}|A'| \gtrsim \de^{2c_1}|A_1| \gtrsim \de^{4c_1}|G|^{1/2}.
\end{equation}
        Set 
        $$H_0 = H \cap (A \times B).$$

\step{Applying Theorem \ref{thm.babyproj}}
 Choose $$\sigma/2 - 5c_1 < \tau < \sigma/2 + 5c_1$$
        so that
        $$|A| = \de^{-\tau}.$$
        Suppose that we may apply Theorem \ref{thm.babyproj} to $A$ and $Xx_3^{-1}$ that is, there exists $\eta > 0$ so that $A$ is a $(\de,\eta,\demm)$-set. We know that since $|x_3| > \de^{\e/2s}$ the directions $ Xx_3^{-1}$  remain a $(\de,t,\dem)$-set. We find $x \in X$ so that 
        \begin{equation}
            \cn{A+xx_3^{-1}A} > \de^{-c}|A| > \de^{-c/2}|G|^{1/2}
        \end{equation}
        where $c >0$ is as in the statement of Theorem \ref{thm.babyproj}, and $c_0$ is small enough. We will remark on the dependence of $c$ on various constants at the conclusion of the proof. 
        
        \step{Concluding the proof}
        Set 
$$F = \{(h_1,x_3^{-1}h_2) : (h_1,h_2) \in H_0\} \subset A_1 \times A_2.$$
        By \eqref{eq.bigH}, since we are still assuming the result is false for $x$ as above, we may find $F' \subset F$ with $|F'| > \depp |F|$ so that 
$$\cn{A_1\overset{F'}{+} xA_2} \leq \de^{-c_0}|G|^{1/2}.$$
        Set
        $$H' = \{(f,x_3f') : (f,f') \in F'\} \subset A \times B.$$
        
        Note that for each $(a,a') \in A \times A$ and $(h_1,h_2) \in H'$ we have $(a,a') + t \in H'$ for $$t = (h_1 - a, h_2 - a') \in (A-A) \times (B-A). $$ Therefore 
        \begin{equation}
            1_{A \times A} \leq \frac{1}{|H'|} \int_{(A -A) \times (B-A)} 1_{H' - t} dt,
        \end{equation}
        and so we have
        \begin{equation}
            1_{A + xx_3^{-1}A} \leq \frac{1}{|H'|} \int_{(A -A) \times (B-A)} 1_{\pi_{xx_3^{-1}}(H') - \pi_{xx_3^{-1}}(t)} dt.
        \end{equation}
        Therefore, using \eqref{eq.unaffected}, \eqref{eq.bigH} and the triangle inequality, we obtain
        \begin{align}
            \cn{A + xx_3^{-1}A} &\leq |H'|^{-1}\cn{A-A}\cn{B-A}\cn{A \overset{H'}{+} xx_3^{-1}B}\\
            &\lesssim \frac{ \cn{A+B}^5\cn{A\overset{H'}{+} xx_3^{-1}B}}{|H'||A||B|^2}\\
             &\lesssim \frac{\de^{-100c_1} |A|^{5/2}|B|^{5/2}\cn{A_1\overset{F'}{+} xA_2}}{|H'||A||B|^2}\\
            &\lesssim \de^{-200c_1} |A|.\\
        \end{align}
Take $c_0 >0$ to be so that $c > 1000c_1,$ which means that 
$$\cn{A_1\overset{F'}{+} xA_2} > \de^{-c_0}|G|^{1/2},$$ a contradiction.

\step{Checking the on-concentration of $A$}
This leaves us to check that $A$ is a $(\de,\eta,\demm)$-set for some $\eta > 0.$ Suppose not. Then there exists a $\de < \rho < 1$ and an interval $I$ of length $\rho$ so that 
$$|A \cap I| > \demm \rho^\eta|A|.$$
Note that this forces 
\begin{equation}\label{eq.boundrho}
\rho < \de^{\e/\eta}.
\end{equation}
Using \eqref{eq.bigfibres} and \eqref{eq.boundonA}, we obtain that
$$|\pi^{-1}_{x_1} (I) \cap G| > \de^{10c_1-\e}\rho^\eta|G|^{1/2}|A| > \de^{20c_1} \rho^\eta|G|.$$

Recall that our choice of $x_1$ was not restricted, and after restricting $\rho$ to $\sim \log 1/\de$ values, we may find $$\de < \rho < \de^{\e/\eta}$$ so that for all $x \in X'$ with $|X'| \approx |X|$ we find an interval $I_x$ of length $\rho$ such that
$$|\pi^{-1}_{x} (I_x) \cap G| > \de^{20c_1} \rho^\eta|G| > \rho^{\eta ( 1 + 20c_1/\e)}|G|,$$
where we have used \eqref{eq.boundrho}.

Let $0 < \alp < 1 .$ Since $X'$ is a $(\de,t,\de^{-10\e}),$ using \eqref{eq.boundrho} it follows that 
    $$\cns {\rho^{\alp}} X \geq \de^{10\e} \rho^{-\alpha t} \geq \rho^{-\alpha t + 10 \eta}.$$
So we may find a $\rho^{\alpha}$-separated net 
$$\{x_1,\ldots, x_N\} \subset X'$$
with $N \sim \rho^{-\alpha t + 10 \eta} .$ 
We now claim that we may find $1 \leq i \neq j \leq N$ so that
\begin{equation}\label{eq.claim}
|\pi^{-1}_{x_i} (I_{x_i}) \cap \pi^{-1}_{x_j} (I_j) \cap G| \geq \de^{-10\e}\rho^{(1-\alp)s} |G|.
\end{equation}
If not, then by inclusion-exclusion we have
\begin{align}
\left| \bigcup_{i=1}^N \pi^{-1}_{x_i}(I_{x_i}) \cap G \right| &\geq \sum_{i=1}^N |\pi^{-1}_{x_i}(I_{x_i}) \cap G| - \sum_{i\neq j} |\pi^{-1}_{x_i} (I_{x_i}) \cap \pi^{-1}_{x_j} (I_j) \cap G| \\
&\gtrsim  \rho^{-\alpha t + \eta ( 11 + 20c_1/\e)} |G| - \rho^{-2\alpha t + (1-\alp)s + 20 \eta}\de^{-10\e} |G|\\
&\gtrsim  \rho^{-\alpha t + 12\eta} |G| - \rho^{-2\alpha t + (1-\alp)s + 10 \eta} |G|,
\end{align}
where $\e,\eta$ are taken small enough in terms of $c_1.$ Now select $0 < \alpha < 1,$ and $\eta > 0$ small enough so the above holds, and so that 
$$\frac{(1-\alpha)s - \alpha t}{2} > \eta > 0.$$
This ensures that

$$\left| \bigcup_{i=1}^N \pi^{-1}_{x_i}(I_{x_i}) \cap G \right| \gg |G|,$$
a blatant contradiction. However, since $x_i$ and $x_j$ are $\rho^{\alp}$-separated, this means that 
$$\pi^{-1}_{x_i} (I_{x_i}) \cap \pi^{-1}_{x_j} (I_j),$$
is contained in a ball of radius $\sim \rho^{1-\alpha}$ and therefore
$$|\pi^{-1}_{x_i} (I_{x_i}) \cap \pi^{-1}_{x_j} (I_j) \cap G| \lesssim \demm \rho^{(1-\alp)s}|G| \ll \de^{-10\e}\rho^{(1-\alp)s}|G|,$$
which contradicts \eqref{eq.claim}. This leaves us to conclude that $A$ is a $(\de,\eta,\demm)$-set, for the $\eta > 0$ selected above, for at least one $x \in X.$

\step{Checking dependence of constants}
We are left to check the correct dependence on constants. Note that in general
$$c_0 = c_0(c),$$
$$c = c(\eta, \tau,t,d),$$
$$\tau = \tau(\sigma),$$
$$\eta = \eta(s,t).$$
Therefore 
$$c_0 = c_0(s,\sigma,d,t) > 0,$$
as required. Further, if $E$ is not a non-commutative extension of $\Q_p$ then 
$$c = c(\eta,\tau, t, t/d),$$
giving us
$$c_0 = c_0(s,\sigma,t,t/d).$$

\step{Checking the coordinate change}
Let $L: E^2 \rightarrow E^2$ be the change of basis which maps $(1,x_1) \rightarrow (1,0)$ and $(1,x_2) \rightarrow (0,1).$ By the separation of $x_1$ and $x_2$ it is readily checked that $L$ is bi-Lipschitz with constant $\dem.$ By an abuse of notation, consider the induced map $L$ on $X$ by
$$L(x) = \pi_y \circ \pi_{x}(L(1,x))^{-1} L(1,x),$$
where $\pi_x$ is the projection to the abscissa, and $\pi_y$ is the projection to the ordinate. Note that
$$\de^{\pm O(\e)} \|L(x)\| \pi_x = \pi_{L(x)} \circ (L^T)^{-1}.$$ It is readily checked that $(L^T)^{-1}(G)$ and $L(X)$ retain the desirable non-concentration properties, and so we apply the proof above with $L(X)$ and $(L^T)^{-1}(G).$
\end{steps}
\end{proof}
\section{Proof of Theorem \ref{thm.adultproj2}}
We upgrade the conclusion Theorem \ref{thm.adultproj} readily once we assume that $X$ strongly avoids sub-algebras.
\begin{prop}\label{prop.adultproj}
      Let $d$ be a positive integer. Let $E$ be a normed division algebra over $\Q_p$ or $\R.$ Let $0 < s \leq \sigma  \leq 2d, t > 0.$ There exists a universal $C \geq 1$ and $c_0 = c_0(s,\sigma,d,t) > 0$ so that the following holds for all $\de,\e >0$ small enough.
    
 Let $G \subset  B_{E \times E}(0,1)$ be a $(\de,s,\demm)$-set with $|G| = \de^{-\sigma}.$ Let $X \subset B_E(0,1)$ be a $(\de,t,\demm)$-set which strongly avoids sub-algebras. Then there exists $F \subset G$ with $|F| > \de^{Cc_0} |G|$ and $Y \subset X$ with $|Y| > (1-\de^\e)|X|$ so that for all $x \in Y$ we have
        \begin{equation}\label{eq.conclusion}
        \cn{\pi_x(F')} > \de^{-c}|G|^{1/2} \text{ for all } F' \subset F \text{ with } |F'| > \depp |F|.
        \end{equation}
        
            Further, provided that $E$ is not a non-commutative division algebra of $\Q_p,$ then $c_0 = c_0(s,\sigma,t,t/d).$
        \end{prop}
The proof is essentially the same as the proof of Theorem \ref{thm.adultproj}. We spell out the main difference below.

\begin{proof}[Proof of Proposition \ref{prop.adultproj}]
The difference is in step 5 only, the other steps are identical. By \eqref{eq.bigH}, since we are still assuming the result is false, in particular for $F$ and any $Y \subset X$ with $|Y| > (1-\de^\e)|X|,$ we perform the following process:
        Let $w_1 \in X$ and $F_1 \subset F$ with $|F_1| > \depp |F|$ so that
        $$\cn{A_1\overset{F_1}{+} w_1A_2} \leq \de^{-c_0}|G|^{1/2}.$$
        Let $w_2 \in X\setminus \{w_1\}$ and $F_2 \subset F$ with $|F_2| > \depp |F|$ so that
        $$\cn{A_1\overset{F_2}{+} w_2A_2} \leq \de^{-c_0}|G|^{1/2},$$
        et cetera. We may continue this process provided that
        $$|X \setminus \{w_1,\dots, w_N\}| > (1-\de^\e)|X|,$$
        that is, we may find $X' \subset X$ with $|X'| > \depp |X|$ so that for all $x \in X$ we may find $F_x \subset F$ with $|F_x| > \depp |F|$ so that
        $$\cn{A_1\overset{F_x}{+} xA_2} \leq \de^{-c_0}|G|^{1/2}.$$
        
         Choose $$\sigma/2 - 5c_1 < \tau < \sigma/2 + 5c_1$$
        so that
        $$|A| = \de^{-\tau}.$$
        Since $X$ strongly avoids sub-algebras it follows that $X'$ avoids sub-algebras. It is readily checked that $X'$ is a $(\de,t, \de^{-10\e})$-set.
        Suppose that we may apply Theorem \ref{thm.babyproj} to $A,$ that is, there exists $\eta > 0$ so that $A$ is a $(\de,\eta,\demm)$-set. We therefore find $x \in X'$ so that 
        \begin{equation}
            \cn{A+xx_3^{-1}A} > \de^{-c}|A|,
        \end{equation}
        where $c >0$ is as in the statement of Theorem \ref{thm.babyproj}. 
        
\end{proof}

\begin{proof}[Proof of Theorem \ref{thm.adultproj2}]

Apply Proposition \ref{prop.adultproj} to find $F_1$ and $X_1$ so that the conclusion holds. If $|G \setminus F_1| < \de^{2\e} |G|$ then we are finished. To see this, take $F \subset G$ with $|F| > \de^\e |G|.$ Then $|F_1 \cap F| \gtrsim \de^\e |F_1|,$ and so $F$ has a large projection for all $x \in X_1.$ Otherwise, we take $F_2 \subset G \setminus F_1$ with $|F_2| > \de^\e |G \setminus F_1| > \de^{3\e} |G|,$ and $X_2 \subset X$ so that the conclusion holds. If $|G \setminus (F_1 \cup F_2)| < \de^{2\e} |G|$ we end the process, otherwise we continue. At stage $N$ we will have found sets $F_j \subset G$ with $|F_j| > \de^{3\e}|G|,$ $X_j \subset X$ with $|X_j| > (1-\de^\e)|X|.$ Since the $F_j$ are disjoint this process must end.

 Let $\mu = 1_X.$ Consider the product measure on $\mu \times \rho$, where $\rho$ is the (non-probability) measure on $\{ 1, \dots, N \}$, where $j$ has weight $|F_j|$. Let $Y$ be the set of $x \in X$ such that 
            \begin{equation}
       	\left|\bigcup_{j: x \in X_j} F_j\right| =\sum_{j: x \in X_j} |F_j|  = \rho\{ j: x \in X_j\} \geq  (1-\delta^{\e})|G|. 
            \end{equation}
We have
        \[
            \int \rho\{ j: x \in X_j \} d\mu(x) = (\mu\times\rho)(\{ (x, j): x \in X_j\}) \ge (1-\delta^{\e})|G|.
        \]    
        By Markov's inequality, this implies that
        \[
        \mu(Y) \geq 1-\delta^{\e},
        \]  
        and so $|Y| \geq (1-\de^\e)|X|.$
        Finally, if $y\in Y$ and $F\subset G$ satisfies $|F|\geq \delta^{\e}|G|$, then there is $j$ such that $x\in X_j$ and $|F\cap F_j|\geq \delta^{2\e}|F_j|.$

\end{proof}

\bibliographystyle{alpha}
	\bibliography{references}

@article{BFLM11,
  author     = {Bourgain, Jean and Furman, Alex and Lindenstrauss, Elon and
                Mozes, Shahar},
  title      = {Stationary measures and equidistribution for orbits of
                nonabelian semigroups on the torus},
  journal    = {J. Amer. Math. Soc.},
  fjournal   = {Journal of the American Mathematical Society},
  volume     = {24},
  year       = {2011},
  number     = {1},
  pages      = {231--280},
  issn       = {0894-0347},
  mrclass    = {37A45 (11L07 22F10 37A17)},
  mrnumber   = {2726604},
  mrreviewer = {Thomas Ward},
  doi        = {10.1090/S0894-0347-2010-00674-1},
  url        = {https://doi.org/10.1090/S0894-0347-2010-00674-1}
}

@article{BHZ24,
  title   = {Khintchine dichotomy for self-similar measures},
  author  = {B{\'e}nard, Timoth{\'e}e and He, Weikun and Zhang, Han},
  journal = {preprint, arXiv:2409.08061},
  year    = {2024}
}

@article{BenardHe24,
  title   = {Multislicing and effective equidistribution for random walks on some homogeneous spaces},
  author  = {B{\'e}nard, Timoth{\'e}e and He, Weikun},
  journal = {preprint, arXiv:2409.03300},
  year    = {2024}
}

@article{He20,
  author   = {He, Weikun},
  title    = {Orthogonal projections of discretized sets},
  journal  = {J. Fractal Geom.},
  fjournal = {Journal of Fractal Geometry. Mathematics of Fractals and
              Related Topics},
  volume   = {7},
  year     = {2020},
  number   = {3},
  pages    = {271--317},
  issn     = {2308-1309},
  mrclass  = {28A80 (11B30)},
  mrnumber = {4148151},
  doi      = {10.4171/jfg/92},
  url      = {https://doi.org/10.4171/jfg/92}
}

@article{OSW24,
  author     = {Orponen, Tuomas and Shmerkin, Pablo and Wang, Hong},
  title      = {Kaufman and {F}alconer estimates for radial projections and a
                continuum version of {B}eck's theorem},
  journal    = {Geom. Funct. Anal.},
  fjournal   = {Geometric and Functional Analysis},
  volume     = {34},
  year       = {2024},
  number     = {1},
  pages      = {164--201},
  issn       = {1016-443X,1420-8970},
  mrclass    = {28A80 (28A78)},
  mrnumber   = {4706445},
  mrreviewer = {Jonathan\ MacDonald\ Fraser},
  doi        = {10.1007/s00039-024-00660-3},
  url        = {https://doi.org/10.1007/s00039-024-00660-3}
}

@article{OrponenShmerkin23,
  author     = {Orponen, Tuomas and Shmerkin, Pablo},
  title      = {On the {H}ausdorff dimension of {F}urstenberg sets and
                orthogonal projections in the plane},
  journal    = {Duke Math. J.},
  fjournal   = {Duke Mathematical Journal},
  volume     = {172},
  year       = {2023},
  number     = {18},
  pages      = {3559--3632},
  issn       = {0012-7094},
  mrclass    = {28A80 (28A75 28A78)},
  mrnumber   = {4718435},
  mrreviewer = {Jonathan MacDonald Fraser},
  doi        = {10.1215/00127094-2022-0103},
  url        = {https://doi.org/10.1215/00127094-2022-0103}
}

@article{OrponenShmerkin23b,
  title   = {Projections, {F}urstenberg sets, and the {ABC} sum-product problem},
  author  = {Orponen, Tuomas and Shmerkin, Pablo},
  journal = {arXiv preprint arXiv:2301.10199},
  year    = {2023}
}

@article {RazZahl24,
    AUTHOR = {Raz, Orit E. and Zahl, Joshua},
     TITLE = {On the dimension of exceptional parameters for nonlinear
              projections, and the discretized {E}lekes-{R}\'{o}nyai theorem},
   JOURNAL = {Geom. Funct. Anal.},
  FJOURNAL = {Geometric and Functional Analysis},
    VOLUME = {34},
      YEAR = {2024},
    NUMBER = {1},
     PAGES = {209--262},
      ISSN = {1016-443X},
   MRCLASS = {28A78 (11B30 28A80 52C10)},
  MRNUMBER = {4706447},
MRREVIEWER = {Esa J\"{a}rvenp\"{a}\"{a}},
       DOI = {10.1007/s00039-024-00664-z},
       URL = {https://doi.org/10.1007/s00039-024-00664-z},
}

@article{Ren23,
  title   = {Discretized Radial Projections in $\mathbb{R}^d$},
  author  = {Ren, Kevin},
  journal = {arXiv preprint arXiv:2309.04097},
  year    = {2023}
}

@article{RenWang23,
  title   = {Furstenberg sets estimate in the plane},
  author  = {Kevin Ren and Hong Wang},
  journal = {arXiv preprint arXiv:2308.08819},
  year    = {2023}
}

@article{Shmerkin23,
  author     = {Shmerkin, Pablo},
  title      = {A non-linear version of {B}ourgain's projection theorem},
  journal    = {J. Eur. Math. Soc. (JEMS)},
  fjournal   = {Journal of the European Mathematical Society (JEMS)},
  volume     = {25},
  year       = {2023},
  number     = {10},
  pages      = {4155--4204},
  issn       = {1435-9855},
  mrclass    = {28A75 (05D99 26A16 28A80 49Q15)},
  mrnumber   = {4634691},
  mrreviewer = {Lars Olsen},
  doi        = {10.4171/jems/1283},
  url        = {https://doi.org/10.4171/jems/1283}
}

@article {ShmerkinWang25,
    AUTHOR = {Shmerkin, Pablo and Wang, Hong},
     TITLE = {On the distance sets spanned by sets of dimension {$d/2$} in
              {$\Bbb R^d$}},
   JOURNAL = {Geom. Funct. Anal.},
  FJOURNAL = {Geometric and Functional Analysis},
    VOLUME = {35},
      YEAR = {2025},
    NUMBER = {1},
     PAGES = {283--358},
      ISSN = {1016-443X},
   MRCLASS = {28A78 (28A80)},
  MRNUMBER = {4865135},
MRREVIEWER = {Lars Olsen},
       DOI = {10.1007/s00039-024-00696-5},
       URL = {https://doi.org/10.1007/s00039-024-00696-5},
}

@article{WangZahl25a,
  title   = {Sticky {K}akeya sets and the sticky {K}akeya conjecture},
  author  = {Wang, Hong and Zahl, Joshua},
  journal = {J. Amer. Math. Soc.},
  volume  = {accepted},
  year    = {2025}
}

@article{WangZahl25b,
  title   = {The {A}ssouad dimension of {K}akeya sets in $\mathbb{R}^3$ },
  author  = {Wang, Hong and Zahl, Joshua},
  journal = {Invent. Math.},
  volume  = {accepted},
  year    = {2025}
}

@article{WangZahl25c,
  author  = {Wang, Hong and Zahl, Joshua},
  journal = {Preprint, arXiv:2502.17655},
  year    = {2025}
}

@article {bou,
	AUTHOR = {Bourgain, J.},
	TITLE = {The discretized sum-product and projection theorems},
	JOURNAL = {J. Anal. Math.},
	FJOURNAL = {Journal d'Analyse Math\'{e}matique},
	VOLUME = {112},
	YEAR = {2010},
	PAGES = {193--236},
	ISSN = {0021-7670,1565-8538},
	MRCLASS = {11B30},
	MRNUMBER = {2763000},
	MRREVIEWER = {Serge\u{\i}\ V.\ Konyagin},
	DOI = {10.1007/s11854-010-0028-x},
}

@article {bougam,
AUTHOR = {Bourgain, Jean and Gamburd, Alex},
TITLE = {On the spectral gap for finitely-generated subgroups of {$\rm
SU(2)$}},
JOURNAL = {Invent. Math.},
FJOURNAL = {Inventiones Mathematicae},
VOLUME = {171},
YEAR = {2008},
NUMBER = {1},
PAGES = {83--121},
ISSN = {0020-9910,1432-1297},
MRCLASS = {22E30 (11B75 43A75)},
MRNUMBER = {2358056},
MRREVIEWER = {Ben\ Joseph\ Green},
DOI = {10.1007/s00222-007-0072-z},
URL = {https://0-doi-org.pugwash.lib.warwick.ac.uk/10.1007/s00222-007-0072-z},
}

@article {bourgaingamspectral,
    AUTHOR = {Bourgain, J. and Gamburd, A.},
     TITLE = {A spectral gap theorem in {${\rm SU}(d)$}},
   JOURNAL = {J. Eur. Math. Soc. (JEMS)},
  FJOURNAL = {Journal of the European Mathematical Society (JEMS)},
    VOLUME = {14},
      YEAR = {2012},
    NUMBER = {5},
     PAGES = {1455--1511},
      ISSN = {1435-9855,1435-9863},
   MRCLASS = {22E30 (11B30)},
  MRNUMBER = {2966656},
MRREVIEWER = {B.\ Sury},
       DOI = {10.4171/JEMS/337},
       URL = {https://doi.org/10.4171/JEMS/337},
}

@article{garaev2007explicit,
  title={An explicit sum-product estimate in $\mathbb{F}_p$},
  author={Garaev, Moubariz Z},
  journal={International Mathematics Research Notices},
  volume={2007},
  pages={rnm035},
  year={2007},
  publisher={Oxford University Press}
}

@article {gut,
	AUTHOR = {Guth, Larry and Katz, Nets Hawk and Zahl, Joshua},
	TITLE = {On the discretized sum-product problem},
	JOURNAL = {Int. Math. Res. Not. IMRN},
	FJOURNAL = {International Mathematics Research Notices. IMRN},
	YEAR = {2021},
	NUMBER = {13},
	PAGES = {9769--9785},
	ISSN = {1073-7928},
	MRCLASS = {11B30 (05D10 28A99)},
	MRNUMBER = {4283564},
	MRREVIEWER = {Hanbin Zhang},
	DOI = {10.1093/imrn/rnz360},
	URL = {https://doi.org/10.1093/imrn/rnz360},
}

@inbook{nic,
author = {de Saxcé, Nicolas},
year = {2025},
month = {04},
pages = {347-372},
title = {A Presentation of the Discretized Sum-Product in Division Algebras},
isbn = {978-3-031-80452-6},
doi = {10.1007/978-3-031-80453-3_13}
}

@article{o2025simple,
  title={Simple proofs of discretised projection theorems},
  author={O'Regan, William and Shmerkin, Pablo and Wang, Hong},
  journal={arXiv preprint arXiv:2511.21656},
  year={2025}
}

@book {tao,
	AUTHOR = {Tao, Terence and Vu, Van H.},
	TITLE = {Additive combinatorics},
	SERIES = {Cambridge Studies in Advanced Mathematics},
	VOLUME = {105},
	NOTE = {Paperback edition [of MR2289012]},
	PUBLISHER = {Cambridge University Press, Cambridge},
	YEAR = {2010},
	PAGES = {xviii+512},
	ISBN = {978-0-521-13656-3},
	MRCLASS = {11-02 (05-02 05D10 11B13 11B30 11P70 11P82 28D05)},
	MRNUMBER = {2573797},
}
    \end{document}